\documentclass[11pt,twoside, final]{amsart}
\copyrightinfo{0}{Iranian Mathematical Society}
\pagespan{1}{\pageref*{LastPage}}
\usepackage{etoolbox,lastpage}
\commby{}
\date{\scriptsize   Received: , Accepted: .}
\usepackage{amsmath,amsthm,amscd,amsfonts,amssymb,enumerate}
\usepackage{graphicx}		
\usepackage{color}
\usepackage[colorlinks]{hyperref}

\newtheorem{theorem}{Theorem}[section]

\newtheorem{lemma}[theorem]{Lemma}
\newtheorem{corollary}[theorem]{Corollary}
\theoremstyle{definition}

\theoremstyle{remark}

\numberwithin{equation}{section}

\def\B{\mathcal{B}}
\def\D{\mathbb{D}}

\def\T{\partial \D}

\def\msk{\medskip}

\def\H{\mathcal H}

\begin{document}

\title[Generalized Hilbert matrix]
{Essential norm of  generalized Hilbert matrix from Bloch type spaces to BMOA and Bloch space}
\author{Songxiao Li and Jizhen Zhou}
\address{Songxiao Li: Institute of Fundamental and Frontier Sciences, University of Electronic Science and Technology of China,
610054, Chengdu, Sichuan, P.R. China\newline
Institute of Systems Engineering, Macau University of Science and Technology, Avenida Wai Long, Taipa, Macau. }  \email{jyulsx@163.com (S. Li)}

\vskip 3mm
\address{Jizhen Zhou: School of Sciences\\
Anhui University of Science and Technology\\
 Huainan, Anhui 232001,  P.R.China}\email{hope189@163.com(J. Zhou)}


  \thanks{The first author  was partially supported by NSF of China (No. 11471143 and No. 11720101003).  The second author was  supported by NSF of Anhui Province, grant No. 1608085MA01.}
%

 \maketitle
%

\begin{abstract}
 Let $\mu$ be a positive Borel measure on the interval $[0,1)$. The Hankel matrix $\H_\mu=(\mu_{n+k})_{n,k\geq 0}$ with entries $\mu_{n,k}=\mu_{n+k}$ induces the operator
$$
\H_\mu(f)(z)=\sum^\infty_{n=0}\left(\sum^\infty_{k=0}\mu_{n,k}a_k\right)z^n
$$
on the space of all analytic functions $f(z)=\sum^\infty_{n=0}a_nz^n$ in the unit disk $\D$.
In this paper, we characterize the boundedness and compactness of $\H_\mu$   from Bloch type spaces to the BMOA and the Bloch space. Moreover we obtain the essential norm of $\H_\mu$ from $\B^\alpha$ to $\B$ and BMOA.
\end{abstract}

\section{\bf Introduction}
Denote by $H(\D)$ the space of all analytic functions on the unit disk $\D=\{z:|z|<1\}$ in the complex plane. For $0 < p \leq \infty$, we let $H^p$ denote the classical Hardy space.  If $f\in H(\D)$ and
$$
\|f\|_{BMOA}=|f(0)|+\sup_{a\in\D}\|f\circ\varphi_a-f(a)\|_{H^2}<\infty,
$$
then we say that $f\in BMOA$. Here $\varphi_a(z)=\frac{a-z}{1-\bar{a}z}, ~a\in \D, $
is a M\"obius transformation of $\D$. Fefferman's duality theorem says that $BMOA=(H^1)^*$. We refer \cite{g} about the theory of $BMOA$.

Let $0<\alpha<\infty$. An $f\in H(\D)$ is said to belong to the Bloch type space (or called the $\alpha-$Bloch space), denoted by  $\B^\alpha$,  if
$$
\|f\|_{\B^\alpha}=\sup_{z\in\D}|f'(z)|(1-|z|^2)^\alpha<\infty.
$$
The classical Bloch space $\B$ is just $\B^1$. It is clear that $\B^\alpha$ is a Banach space with the norm
$\|f\|=|f(0)|+\|f\|_{\B^\alpha}$. See \cite{zhu1} for the theory of Bloch type spaces.

For a subarc $I\subset \partial\mathbb{D}$, let $S(I)$ be the Carleson box based on $I$ with
$$S(I)=\{z\in \D: 1-|I|\leq |z|<1 \ \ and\  \frac{z}{|z|}\in I\}.$$
If $I=\T$, let $S(I)=\D$.   For $0<s<\infty$, we say that a positive Borel measure $\mu$ is an $s-$Carleson measure
on $\mathbb{D}$ if (see \cite{Du})
$$
\|\mu\|=\sup_{I\subset\partial\mathbb{D}}\frac{\mu(S(I))}{|I|^s}<\infty.
$$
We say that a positive Borel measure $\mu$ is a vanishing $s-$Carleson measure on $\D$ if
$$
\lim_{|I|\rightarrow 0}\frac{\mu(S(I))}{|I|^s}=0.
$$
Here and henceforth $\sup_{I\subset\partial\mathbb{D}}$ indicates the supremum taken over all subarcs $I$ of $\partial\mathbb{D}$. $|I|=(2\pi)^{-1}\int_I|d\xi|$ is the normalized length of the subarc $I$ of the unit circle $\T$. When $s=1$,   $\mu$ is called a Carleson measure on $\mathbb{D}$. It is well known that $\mu$ is a Carleson measure if and only if
$$
\left(\int_\D|f(z)|^pd\mu(z)\right)^{1/p}\leq\|\mu\|\|f\|_{H^p},
$$
for any $f\in H^p, 0<p<\infty$. See, for example,  \cite{G}.

A positive Borel measure $\mu$ on $[0,1)$ can be seen as a Borel measure on $\D$ by identifying it  with measure $\widetilde{\mu}$ defined by
$$
\widetilde{\mu}(E)=\mu(E\cap[0,1))
$$
for any Borel subset $E$ of $\D$. Then a positive Borel measure $\mu$ on [0,1) is an $s$-Carleson measure if there exists a constant $C>0$ such that (see \cite{gm})
$$
\mu([t,1))\leq C(1-t)^s.
$$
A vanishing $s-$Carleson measure on $[0,1) $ can be defined similarly.

Let $\mu$ be a finite positive measure on $[0,1)$ and $n=0,1,2,\cdots$. Denote $\mu_n$ the moment of order $n$ of $\mu$, that is,
$\mu_n=\int_{[0,1)} t^nd\mu(t).$  Let $\H_{\mu}$ be
the Hankel matrix $(\mu_{n,k})_{n,k\geq 0}$ with entries  $ \mu_{n,k}= \mu_{n+k}$.  The matrix $\H_{\mu}$ induces an operator, denoted also by  $\H_{\mu}$,   on $H(\D)$ by its action on the Taylor coefficient:
$$a_n\rightarrow\sum^\infty_{k=0}\mu_{n,k}a_k, n=0,1,2,\cdots.$$
 More precisely, if $f(z)=\sum^\infty_{k=0}a_kz^k\in H(\D)$,   then
$$
\H_\mu(f)(z)=\sum^\infty_{n=0}\left(\sum^\infty_{k=0}\mu_{n,k}a_k\right)z^n,    \eqno(1.5)
$$
whenever the right hand side makes sense and defines an analytic function in $\D$.

As in \cite{gp}, to obtain an integral representation of $\H_{\mu}$, we  write
\begin{equation}\label{gs1.1}
I_\mu(f)(z)=\int_{[0,1)}\frac{f(t)}{1-tz}d\mu(t),
\end{equation}
whenever the right hand side makes sense and defines an analytic function in $\D$.

If $\mu$ is the Lebesgue measure on $[0, 1)$, then the matrix $\H_\mu$ is just the classical Hilbert matrix $H = \big(\frac{1}{n+k+1}  \big)_{n,k\geq  0}$, which induces the classical Hilbert operator $H$. The Hilbert operator $H$ was studied in \cite{ams, bw, di,ds,djv,lnp}. A generalized Hilbert operator was studied in \cite{LS}.

The operator $\H_\mu$ acting on analytic functions spaces has been studied by many authors. Galanopoulos  and Pel\'aez  \cite{gp}
obtained a characterization that $\H_\mu$ is bounded or compact on $H^1$. Chatzifountas, Girela and Pel\'aez  \cite{cgp} described the measures $\mu$ for which $\H_\mu$  is bounded (compact) operator from $H^p$ into $H^q, 0<p,q<\infty$. See \cite{gm2} about the Hankel matrix acting on the Dirichlet spaces.

Let $X$ and $Y$ be two Banach spaces. The essential norm of a continuous linear operator $T$ between normed linear spaces $X$ and $Y$ is the distance to set of compact operators $K$, that is
$\|T\|^{X\rightarrow Y}_e=inf\{\|T-K\|: K$ is compact$\}$, where $\|\cdot\|$ is the operator norm. It is easy to see that $\|T\|^{X\rightarrow Y}_e=0$
if and only if $T$ is compact. See \cite{llx,zh} for the study of essential norm of some operators.

In \cite{gm, gm1}, Girela  and Merch\'an  studied the operator $\H_\mu$ acting on spaces of analytic functions on $\D$ such as the Bloch space, BMOA, the Besov spaces and Hardy spaces.  Motivated by \cite{gm, gm1}, in this paper, we characterize the boundedness and compactness of   $\H_\mu$  from $\B^\alpha$ to the $BMOA$ and the Bloch space when $0<\alpha<1$ and $\alpha>1$. Moreover we characterize the essential norm of $\H_\mu$ from $\B^\alpha$ to $\B$ and BMOA.

In this paper, $C$ denotes a constant which may be different in each case.

\section{The operator $\H_\mu:\B^\alpha\rightarrow BMOA$ ($\B$), $0<\alpha<1$}

In this section, we  characterize the boundedness of $\H_\mu$ from  $\B^\alpha$ into the $BMOA$ and the Bloch space when $0<\alpha<1$. For this purpose, we need some auxiliary results.\msk

\begin{lemma}\label{lem1} \cite{zhu1} If $0<\alpha<1$, then $f\in \B^\alpha$ are bounded. If $\alpha>1$, then $f\in\B^\alpha$ if and only if there exists some constant $C$ such that
$$
|f(z)|\leq \frac{C}{(1-|z|^2)^{\alpha-1}}.
$$
\end{lemma}

The following lemma can be found  in \cite{xi1} (see Corollary 3.3.1 in \cite{xi1}).\msk

\begin{lemma}\label{lem2} If $a_n\downarrow 0$ , then $f(z)=\sum^\infty_{n=0}a_nz^n\in\B$ if and only if
$\sup_nna_n<\infty.$
\end{lemma}

\begin{theorem}\label{theorem1} Let $\mu$ be a positive measure on $[0,1)$ and $0<\alpha<1$. Then the following statements are equivalent.
\begin{enumerate}
\item[(1)] The operator $\H_\mu$ is bounded from $\B^\alpha$ into $\B$.

\item[(2)] The operator $\H_\mu$ is compact from $\B^\alpha$ into $\B$.

\item[(3)] The operator $\H_\mu$ is bounded from $\B^\alpha$ into $BMOA$.

\item[(4)] The operator $\H_\mu$ is compact from $\B^\alpha$ into $BMOA$.

\item[(5)] The measure $\mu$ is a Carleson measure.
\end{enumerate}
\end{theorem}

\begin{proof}

 (1)$\Rightarrow$ (5). Assume that the operator $\H_\mu$ is bounded from $\B^\alpha$ into $\B$. Let $f(z)=1\in\B^\alpha, 0<\alpha<1$. Then
$$
\H_\mu(f)(z)=\sum^\infty_{n=0}\left(\sum^\infty_{k=0}\mu_{n,k}a_k\right)z^n=\sum^\infty_{n=0}\mu_{n,0}z^n\in\B.
$$
Note that $\mu_{n,0}$ is positive and decreasing. For any $0<\lambda<1$, we choose $n$ such that $1-\frac{1}{n}<\lambda<1-\frac{1}{n+1}$. Lemma \ref{lem2} gives that
$$
\infty>n\mu_{n,0}=n\int^1_0t^nd\mu(t)\geq n\lambda^n\int^1_\lambda d\mu(t)\geq\frac{\mu([\lambda,1))}{e(1-\lambda)}.
$$
The above estimate gives that $\mu$ is a Carleson measure.

(5)$\Rightarrow$(3). Assume that $\mu$ is a Carleson measure. Lemma \ref{lem1} implies that $\B^\alpha$ is a subspace of $H^1$ for $0<\alpha<1$. Then $\H_\mu$ is well defined for all $z\in\D$ and $\H_\mu(f)$ is an analytic function in $\D$ by Proposition 1.1 in \cite{gp}. Moreover, $\H_\mu(f)=I_\mu(f)$ for any $f\in\B^\alpha,0<\alpha<1$.

Fixed any given $f\in\B^\alpha$ for $0<\alpha<1$,  then
$$\int_{[0,1)}|f(t)|d\mu(t)\leq\|f\|_{\B^\alpha}\int_{[0,1)}d\mu(t)<\infty.$$ Then we have
$$
\int^{2\pi}_0\int_{[0.1)}\left|\frac{f(t)g(e^{i\theta})}{1-rte^{i\theta}}\right|d\mu(t)d\theta<\infty
$$
for any $f\in\B^\alpha, g\in H^1, 0<r<1$ and $0<\alpha<1$. It is easy to obtain that
\begin{equation}\label{gs2.1}
\int^{2\pi}_0I_\mu(f)(re^{i\theta})\overline{g(e^{i\theta})}d\theta=\int_{[0.1)}f(t)\overline{g(rt)}d\mu(t)
\end{equation}
whenever $f\in\B^\alpha, g\in H^1$ and $0<\alpha<1$. The reader can refer the proof of Theorem 2.2 in \cite{gm}.
Using \eqref{gs2.1}, we have
\begin{equation}\nonumber
\begin{split}
\left|\int^{2\pi}_0I_\mu(f)(re^{i\theta})\overline{g(e^{i\theta})}d\theta\right|=&\left|\int_{[0.1)}f_n(t)\overline{g(rt)}d\mu(t)\right|\\
\leq &\|f\|_{\B^\alpha}\int_{[0.1)}|g(rt)|d\mu(t)\\
\leq &\|\mu\|\|f\|_{\B^\alpha}\int^{2\pi}_0|g(re^{i\theta})|d\theta\\
\leq &\|\mu\|\|f\|_{\B^\alpha}\|g\|_{H^1}.\\
\end{split}
\end{equation}
We obtain $\H_\mu(f)=I_\mu(f)\in BMOA$ for any $f\in\B^\alpha$ by Fefferman's duality Theorem.

(5)$\Rightarrow$(4). Assume that $\mu$ is a Carleson measure. Then $\H_\mu$ is bounded from $\B^\alpha$ to $BMOA$ and $\H_\mu(f)=I_\mu(f)$ for any $f\in\B^\alpha,0<\alpha<1$.
Let $\{f_n\}$ be any sequence with $\sup_n\|f_n\|_{\B^\alpha}\leq 1$ and $\lim_{n\rightarrow\infty}f_n(z)=0$ on any compact subset of $\D$.
Then we have $\sup_{z\in\D}|f_n(z)|\rightarrow 0$ as $n\rightarrow\infty$ by Lemma 3.2 in \cite{zhang}.
Applying \eqref{gs2.1} again, we have
\begin{equation}\nonumber
\begin{split}
\left|\int^{2\pi}_0I_\mu(f_n)(re^{i\theta})\overline{g(e^{i\theta})}d\theta\right|=&\left|\int_{[0.1)}f_n(t)\overline{g(rt)}d\mu(t)\right|\\
\leq &\sup_{0<t<1}|f_n(t)|\int_{[0.1)}|g(rt)|d\mu(t)\\
\leq &\sup_{0<t<1}|f_n(t)|\|\mu\|\|g\|_{H^1}.\\
\end{split}
\end{equation}
Then
$$
\lim_{n\rightarrow\infty}\int^{2\pi}_0I_\mu(f_n)(re^{i\theta})\overline{g(e^{i\theta})}d\theta=0.
$$
This prove that $\lim_{n\rightarrow\infty}\H_\mu(f_n)=\lim_{n\rightarrow\infty}I_\mu(f_n)=0$. So $\H_\mu$ is compact.

The other cases are trivial. The proof is complete.
\end{proof}

\begin{corollary}\label{cor1} Let $\mu$ be a positive Borel measure on $[0,1)$. If $\H_\mu$ is bounded from $\B^\alpha$ to $\B$ for any $0<\alpha<1,$ then
$$
\|\H_\mu\|^{\B^\alpha\rightarrow\B}_e=\|\H_\mu\|^{\B^\alpha\rightarrow BMOA}_e =0.
$$

\end{corollary}

\section{\bf The operator $\H_\mu:\B^\alpha\rightarrow BMOA$ ($\B$), $\alpha>1$}
In this section, we will give the essential norm of the operator $\H_\mu$ from $\B^\alpha$ to $BMOA$ and $\B$ for $\alpha>1$. The following lemma will be needed in the proof of the main results.\msk

\begin{lemma}\label{lem3}  Let $\mu$ be a positive Borel measure on $[0,1)$ and $\alpha>1$. Then the following conditions are equivalent.
\begin{enumerate}
\item[(1)] $\int_{[0,1)}(1-t)^{1-\alpha}d\mu(t)<\infty$.

\item[(2)] For any given $f\in\B^\alpha$, the integral in \eqref{gs1.1} converges for all $z\in\D$ and the resulting function $I_\mu(f)$ is analytic on $\D$.
\end{enumerate}
\end{lemma}

\begin{proof} (1)$\Rightarrow$(2). We assume that (1) holds. Lemma \ref{lem1} gives
\begin{equation}\label{gs3.1}
\int_{[0,1)}|f(t)|d\mu(t)\leq C\|f\|_{\B^\alpha}\int_{[0,1)}(1-t^2)^{1-\alpha}d\mu(t)\leq C\|f\|_{\B^\alpha}.
\end{equation}
This implies that
$$
\int_{[0,1)}\frac{|f(t)|}{|1-tz|}d\mu(t)\leq  C\frac{\|f\|_{\B^\alpha}}{1-|z|}
$$
for any $f\in \B^\alpha$ and $z\in\D$. By \eqref{gs3.1} we have
\begin{equation}\label{gs3.2}
\sup_{n\geq 0}\left|\int_{[0,1)}t^nf(t)d\mu(t)\right|<\infty.
\end{equation}
\eqref{gs3.2} and Fubini's Theorem give that the integral $\int_{[0,1)}\frac{f(t)}{1-tz}d\mu(t)$ converges absolutely for any fixed $z\in\D$. Then we have
$$
\int_{[0,1)}\frac{f(t)}{1-tz}d\mu(t)=\sum^\infty_{n=0}\left(\int_{[0,1)}t^nf(t)d\mu(t)\right)z^n, \quad z\in\D.
$$
Hence $I_{\mu}(f)$ is a well defined analytic function in $\D$ and
$$
I_\mu(f)(z)=\sum^\infty_{n=0}\left(\int_{[0,1)}t^nf(t)d\mu(t)\right)z^n, \quad z\in\D.
$$

 (2)$\Rightarrow$(1). Let $f(z)=(1-z)^{1-\alpha}$. Then $f$ belongs to $\B^\alpha$. So $I_{\mu}(f)$ is well defined for every $z\in\D$. In particular,
 $$
 I_\mu(f)(0)=\int_{[0,1)}(1-t)^{1-\alpha}d\mu(t)
 $$
 is a complex number. Since $\mu$ is a positive Borel measure on $[0,1)$, we get the desired result.  The proof is complete.
 \end{proof}

\begin{lemma}\label{lem4}  Let $\mu$ be a positive measure on $[0,1)$ and $\alpha>1$.  Let $v$ be the positive measure on $[0,1)$ defined by
$$dv(t)=(1-t)^{1-\alpha}d\mu(t). $$
Then the following conditions are equivalent.
\begin{enumerate}
\item[(1)] $\mu$ is a $\alpha$-Carleson measure.

\item[(2)] $v$ is a Carleson measure.
\end{enumerate}
\end{lemma}

\begin{proof}
(2)$\Rightarrow$(1) Note that $v([t,1)\lesssim (1-t)$ and $d\mu(t)=(1-t)^{\alpha-1}dv(t)$. We have
$$
\mu([t,1))=\int^1_t(1-s)^{\alpha-1}dv(s)\leq(1-t)^{\alpha-1}\int^1_tdv(s)\lesssim(1-t)^\alpha.
$$
(1)$\Rightarrow$(2) Note that $\mu([t,1))\lesssim (1-t)^\alpha$. Integrating by parts, we obtain
\begin{equation}\nonumber
\begin{split}
v([t,1))=&\int^1_t(1-s)^{1-\alpha}d\mu(s)\\
       =&(1-t)^{1-\alpha}\mu([t,1))+(\alpha-1)\int^1_t(1-s)^{-\alpha}\mu([s,1))ds\\
\lesssim &(1-t)+(\alpha-1)\int^1_tds\\
\lesssim &(1-t).
\end{split}
\end{equation}
The proof is complete.
\end{proof}

\vskip 2mm
\begin{lemma}\label{lem5} Let $f(z)=\sum^\infty_{n=0}a_nz^n \in \B^\alpha$  for any $\alpha>0$. Then
\begin{equation}\label{gs3.3}
\sup_{n}\sum^{2^{n+1}}_{k=2^n+1}\left|\frac{a_k}{k^{\alpha-1}}\right|^2<C\|f\|^2_{\B^\alpha}.
\end{equation}
\end{lemma}
\begin{proof} For any $0<r<1$ and $f(z)=\sum^\infty_{k=0}a_k z^k\in\B^\alpha$, we have
$$
(1-r)^{2\alpha}\int^{2\pi}_0|f'(re^{i\theta})|^2d\theta\leq\|f\|^2_{\B^\alpha}.
$$
This gives that
$$
\sum^\infty_{k=1}k^2|a_k|^2r^{2k}\leq \|f\|^2_{\B^\alpha}(1-r)^{-2\alpha}.
$$
Choosing $r=1-2^{-n}$ for any fixed $n$, we obtain
\begin{equation}\label{gs3.4}
\sum^{2^{n+1}}_{k=2^n+1}k^2|a_k|^2(1-2^{-n})^{2k}\leq \|f\|^2_{\B^\alpha}2^{2\alpha n}.
\end{equation}
Then  \eqref{gs3.3} follows by \eqref{gs3.4}.
\end{proof}
\vskip 2mm

A complex sequence $\{\lambda_n\}^\infty_{n=0}$ is a multiplier from $l(2, \infty)$ to $l^1$ if and only if there exists a positive constant $C$
such that whenever $\{a_n\}^\infty_{n=0}\in l(2, \infty)$, we have $\sum^\infty_{n=0}|\lambda_na_n|\leq C\|\{a_n\}\|_{l(2, \infty)}$. $l(2,\infty)$ consists all the sequences $\{b_k\}^\infty_{k=0}$ for which
$$
\left\{\left(\sum^{2^{n+1}}_{k=2^n+1}|b_k|^2\right)^{1/2}\right\}^\infty_{n=0}\in l^\infty .
$$
The following result can be found in \cite{ke}.\\

\begin{lemma}\label{lem6} A complex sequence $\{\lambda_n\}^\infty_{n=0}$ is a multiplier from $l(2, \infty)$ to $l^1$ if and only if
$$
\sum^\infty_{n=1}\left(\sum^{2^{n+1}}_{k=2^n+1}|\lambda_k|^2\right)^{1/2}<\infty.
$$
\end{lemma}

\begin{theorem}\label{theorem2} Let $\mu$ be a positive measure on $[0,1)$ and $\alpha>1$. Then the following statements are equivalent.
\begin{enumerate}
\item[(1)] The measure $\mu$ is an $\alpha$-Carleson measure.

\item[(2)] The operator $\H_\mu$ is bounded from $\B^\alpha$ into $\B$.

\item[(3)] The operator $\H_\mu$ is bounded from $\B^\alpha$ into BMOA.
\end{enumerate}
\end{theorem}

\begin{proof}
(3)$\Rightarrow$ (2). It is trivial.

  (2)$\Rightarrow$(1).  We suppose that $\H_\mu$ is bounded from $\B^\alpha$ into $\B$ for $\alpha>1$. For any $0<\lambda<1$, let
\begin{equation}\label{gs3.5}
f_\lambda(z)=\frac{1-\lambda^2}{(1-\lambda z)^\alpha}=\sum^\infty_{k=0}a_{k,\lambda}z^n,
\end{equation}
where $a_{k,\lambda}=O(k^{\alpha-1}\lambda^k)$. It is easy to see that $f_\lambda\in\B^\alpha.$ Then
$$
\H_\mu(f)(z)=\sum^\infty_{n=0}\left(\sum^\infty_{k=0}\mu_{n,k}a_k\right)z^n\in \B.
$$
Lemma \ref{lem2} gives that
\begin{equation}\nonumber
\begin{split}
\infty>&\sup_nn\sum^\infty_{k=0}\mu_{n,k}a_{k,\lambda}\\
=&\sup_nn(1-\lambda^2)\sum^\infty_{k=0}k^{\alpha-1}\lambda^k\int^1_0t^{n+k}d\mu(t)\\
\geq&\sup_n n(1-\lambda^2)\sum^\infty_{k=0}k^{\alpha-1}\lambda^k\int^1_\lambda t^{n+k}d\mu(t)\\
\geq&\sup_nn(1-\lambda^2)\lambda^n\mu([\lambda,1))\sum^\infty_{k=0}k^{\alpha-1}\lambda^{2k}\\
=&\sup_n n\lambda^n\frac{1-\lambda^2}{(1-\lambda^2)^\alpha}\mu([\lambda,1)).
\end{split}
\end{equation}
We choose $n$ such that $1-\frac{1}{n+1}\leq\lambda<1-\frac{1}{n}.$ We have
\begin{equation}\label{gs3.6}
\infty>\frac{1}{e(1-\lambda^2)^\alpha}\mu([\lambda,1)).
\end{equation}
So $\mu$ is an $\alpha-$Carleson measure.

(1)$\Rightarrow$(3). Assume that the condition (1) holds.
Lemma \ref{lem3} shows that $I_\mu(f)$ is analytic on $\D$. Let $f(z)=\sum^\infty_{n=0}a_nz^n\in\B^\alpha.$
By Lemma \ref{lem5} we  have the sequence $\{a_k/k^{\alpha-1}\}\in l(2,\infty)$.
Since $\mu$ is a $\alpha$-Carleson measure, we have $\mu_{k}\leq\frac{C}{k^\alpha}$ by Lemma 2.7 in \cite{gm}.
There exists a constant $C$ such that
$$
\sum^\infty_{n=1}\left(\sum^{2^{n+1}}_{k=2^n+1}(\mu_kk^{\alpha-1})^2\right)^{1/2}\lesssim
\sum^\infty_{n=1}\left(\sum^{2^{n+1}}_{k=2^n+1}\frac{1}{k^2}\right)^{1/2}
\lesssim\sum^\infty_{n=1}\frac{1}{2^{n/2}}<\infty.
$$
This shows that the sequence
$\{\mu_k k^{\alpha-1}\}$ is a multiplier from $l(2, \infty)$ to $l^1$ by Lemma \ref{lem6}. Note that $\{\mu_n\}^\infty_{n=1}$  is  a decreasing sequence of positive numbers. Given any $f(z)=\sum^\infty_{n=0}a_nz^n\in\B^\alpha$ for $\alpha>1$, we have
\begin{equation}\nonumber
\begin{split}
\sum^\infty_{k=1}|\mu_{n+k}a_k|&\leq\sum^\infty_{k=1}|\mu_{k}a_k|\leq\sum^\infty_{k=1}\frac{\mu_{k}}{k^{1-\alpha}}\frac{|a_k|}{k^{\alpha-1}}\\
&\leq C\sup_n\left(\sum^{2^{n+1}-1}_{k=2^n}\frac{|a_k|^2}{k^{2(\alpha-1)}}\right)^{1/2}<C\|f\|_{\B^\alpha}.
\end{split}
\end{equation}
This implies that $\H_\mu(f)(z)$ is well defined for all $z\in\D$ and $\H_\mu(f)$ is an analytic function in $\D$. Applying Fubini's Theorem, we get
\begin{equation}\nonumber
\begin{split}
\H_\mu(f)(z)=&\sum^\infty_{n=0}\left(\sum^\infty_{k=0}\mu_{n+k}a_k\right)z^n=\sum^\infty_{k=0}a_k\left(\sum^\infty_{n=0}\mu_{n+k}z^n\right)\\
=&\sum^\infty_{k=0}a_k\left(\sum^\infty_{n=0}\int_{[0,1)}t^{n+k}z^nd\mu(t)\right)\\
=&\sum^\infty_{k=0}\int_{[0,1)}\left(\sum^\infty_{n=0} t^nz^n\right)a_kt^k d\mu(t)\\
=&\int_{[0,1)}\sum^\infty_{k=0}\frac{a_kt^k}{1-tz}d\mu(t)=I_\mu(f)(z).
\end{split}
\end{equation}

Note that $|f(t)|\lesssim(1-t)^{1-\alpha}$ by Lemma \ref{lem1}. Applying \eqref{gs2.1} and Lemma \ref{lem4}, we have
\begin{equation}\nonumber
\begin{split}
\left|\int^{2\pi}_0I_\mu(f)(re^{i\theta})\overline{g(e^{i\theta})}d\theta\right|=&\left|\int_{[0.1)}f(t)\overline{g(rt)}d\mu(t)\right|\\
\leq &\|f\|_{\B^\alpha}\int_{[0.1)}|g(rt)|(1-t)^{1-\alpha}d\mu(t)\\
\leq &\|\mu\|\|f\|_{\B^\alpha}\int^{2\pi}_0|g(re^{i\theta})|d\theta\\
\leq &\|\mu\|\|f\|_{\B^\alpha}\|g\|_{H^1}.\\
\end{split}
\end{equation}
We obtain $\H_\mu(f)=I_\mu(f)\in BMOA$ by Fefferman's duality Theorem for any $f\in\B^\alpha, \alpha>1$. The proof is complete.
\end{proof}

\begin{theorem}\label{theorem3} Let $\mu$ be a positive measure on $[0,1)$ and $\alpha>1$. Then the following statements are equivalent.
\begin{enumerate}
\item[(1)] The measure $\mu$ is a vanishing $\alpha$-Carleson measure.

\item[(2)] The operator $\H_\mu$ is compact from $\B^\alpha$ spaces into $\B$.

\item[(3)] The operator $\H_\mu$ is compact from $\B^\alpha$ spaces into $BMOA$.

\end{enumerate}
\end{theorem}

\begin{proof}
(3)$\Rightarrow$(2). It is trivial.

(2)$\Rightarrow$(1). Suppose that $\H_\mu:\B^\alpha\rightarrow \B$ is compact. Let $f_{\lambda}$ be defined by \eqref{gs3.5}.
Then $\{f_{\lambda}\}$ is a bounded sequence in $\B^\alpha$ and $\lim_{n\rightarrow\infty}f_{\lambda_n}(z)=0$ on any compact subset of $\D$.
Then we have
$$
\lim_{n\rightarrow \infty}\|\H_\mu(f_{\lambda_n})\|_{\B^\alpha}=0.
$$
The proof of Theorem \ref{theorem2} gives that
$$
\|\H_\mu(f_{\lambda_n})\|_{\B^\alpha}\geq\frac{\mu([\lambda_n,1))}{e(1-\lambda_n^2)}.
$$
Consequently, $\mu$ is a vanishing Carleson measure by Lemma \ref{lem4}.

(1)$\Rightarrow$(3). Assume that $\mu$ is a vanishing $\alpha$-Carleson measure. The proof of the sufficiency for the boundedness gives that
$\H_\mu(f)=I_\mu(f)$ and
$$
\left|\int^{2\pi}_0 \H_\mu(f)(e^{i\theta})\overline{g(re^{i\theta})}d\theta\right|\leq\int_{[0.1)}|f(t)g(rt)|d\mu(t)
$$
for all $f\in\B^\alpha$ and $g\in H^1$.
Let $f_n$ be any sequence with $\sup_n\|f_n\|_{\B^\alpha}\leq 1$ and $\lim_{n\rightarrow\infty}f_n(z)=0$ on any compact subset of $\D$.
Then we have
\begin{equation}\label{gs3.7}
\int_{[0,r)}|f_n(t)g(rt)|d\mu(t)=0.
\end{equation}
Since  $v$ is a vanishing Carleson measure, where $v$ is defined by $dv(t)=\log\frac{e}{1-t}d\mu(t)$. we obtain
\begin{equation}\label{gs3.8}
\int_{[r,1)}|f_n(t)g(rt)|d\mu(t)\leq \int_{[0,1)}|g(rt)|dv_r(t)<\|v-v_r\|\|g\|_{H^1},
\end{equation}
where $dv_r(t)=\chi_{0<t<r}dv(t)$. It is well known that $v$ is a vanishing Carleson measure if and only if
$$\|v-v_r\|\rightarrow 0, r\rightarrow 1.$$
See p. 283 of \cite{zhu2}. Combining \eqref{gs3.7} and \eqref{gs3.8}, then
$$
\lim_{n\rightarrow\infty}\left(\lim_{r\rightarrow 1}\int_{[0,1)}|f_n(t)g(rt)|d\mu(t)\right)=0.
$$
This prove that $\lim_{n\rightarrow\infty}\H_\mu(f_n)=0$. So $\H_\mu$ is compact. The proof is complete.
\end{proof}

\begin{theorem}\label{theorem4}Let $\mu$ be a positive measure on $[0,1)$. If $\H_\mu$ is bounded from $\B^\alpha$ to $\B$ for any $\alpha>1$, then
\begin{equation}\label{gs3.10}
\|\H_\mu\|^{\B^\alpha\rightarrow\B}_e\approx\|\H_\mu\|^{\B^\alpha\rightarrow BMOA}_e\approx\limsup_{r\rightarrow 1^-}\frac{\mu([r,1))}{(1-r)^\alpha}.
\end{equation}
\end{theorem}

\begin{proof}
 For any $f\in\B^\alpha$, we have
$$
\|\H_\mu(f)\|^{\B^\alpha\rightarrow\B}\lesssim\|\H_\mu(f)\|^{\B^\alpha\rightarrow BMOA}.
$$
This gives that
$$
\|\H_\mu\|^{\B^\alpha\rightarrow\B}_e\lesssim\|\H_\mu\|^{\B^\alpha\rightarrow BMOA}_e.
$$
We now give the upper estimate of $\H_\mu$ from $\B^\alpha$ to BMOA. Since $\H_\mu$ is bounded from $\B^\alpha$ to $\B$, then the operator $\H_\mu$ from $\B^\alpha$ to BMOA is bounded and $\mu$ is an $\alpha-$Carleson measure by Theorem \ref{theorem2}. For any $0<r<1$, the positive measure $\mu_r$ is defined by
\begin{equation}\label{gs3.9}
\mu_r(t)=\begin{cases}
\mu(t), & \enspace 0\leq t\leq r, \\
0, &\enspace r<t<1, \end{cases}
\end{equation}
It is easy to check that $\mu_r$ is a vanishing $\alpha$-Carleson measure. We have $\H_{\mu_r}$ is compact from $\B^\alpha$ to BMOA by Theorem \ref{theorem3}. Then
\begin{equation}\label{gs3.10}
\|\H_\mu-\H_{\mu_r}\|^{\B^\alpha\rightarrow BMOA}=\inf_{\|f\|_{\B^\alpha}=1}\|\H_{\mu-\mu_r}(f)\|_{BMOA}.
\end{equation}
By (\ref{gs2.1}) we have
\begin{equation}\nonumber
\begin{split}
\left|\int^{2\pi}_0\H_\mu(f_n)(re^{i\theta})\overline{g(e^{i\theta})}d\theta\right|
\leq&\int_{[0.1)}\left|\overline{g(rt)}\right|(1-t)^{1-\alpha}d(\mu-\mu_r)(t)\\
\leq&\|v-v_r\|\|g\|_{H^1},
\end{split}
\end{equation}
for any $g\in H^1$, where $dv(t)=(1-t)^{1-\alpha}d\mu(t)$ and $dv_r(t)=(1-t)^{1-\alpha}d\mu_r(t)$. The above estimate gives
$$
\|\H_\mu\|^{\B^\alpha\rightarrow BMOA}_e\lesssim\limsup_{r\rightarrow 1^-}\frac{\mu([r,1))}{(1-r)^\alpha}.
$$

We now give the lower estimate of $\H_\mu$ from $\B^\alpha$ to $\B$. For any $0<\lambda<1$, let $f_\lambda$ be defined by \eqref{gs3.5}.
Then $f_\lambda\in\B^\alpha.$
Since $f_\lambda\rightarrow 0$ weakly in $\B^\alpha$,  we have $\|Kf_\lambda\|\rightarrow 0$ as $\lambda\rightarrow 1$ for any
compact operator $K$ on $\B^\alpha$. Moreover
\begin{eqnarray*}
\|\H_\mu-K\|^{\B^\alpha\rightarrow\B}\geq\|(\H_\mu-K)f_\lambda\|_{\B}\geq\|\H_\mu f_\lambda\|_\B-\|Kf_\lambda\|_\B.
\end{eqnarray*}
By the proof of Theorem \ref{theorem2} we have
\begin{equation}\nonumber
\begin{split}
\|\H_{\mu}(f_\lambda)\|_\B\geq\sup_nn\sum^\infty_{k=0}\mu_{n,k}a_{k,\lambda}\geq\sup_n nr^n\frac{1-\lambda^2}{(1-r\lambda)^\alpha}\mu([r,1)).
\end{split}
\end{equation}
Let $r=\lambda$ and we choose $n$ such that $1-\frac{1}{n+1}\leq\lambda<1-\frac{1}{n}$. We have
\begin{equation}
\|\H_{\mu}(f_\lambda)\|_\B>\frac{1}{e(1-\lambda^2)^\alpha}\mu([\lambda,1)).
\end{equation}
Then
$$
\|\H_\mu\|^{\B^\alpha\rightarrow\B}_e\geq\limsup_{\lambda\rightarrow 1^-}\|\H_\mu f_\lambda\|_\B\gtrsim\limsup_{r\rightarrow 1^-}\frac{\mu([r,1))}{(1-r)^\alpha}.
$$
The proof is complete.
\end{proof}

\section{Essential norm of $\H_\mu$ on $\B$}

The reader can refer\cite{gm, gm1} for the results of $\H_\mu:\B\rightarrow BMOA$ and  $\H_\mu:\B\rightarrow\B$. In this section, we  characterize the essential of norm of $\H_\mu$ on $\B$.  The following results will be needed in the proof of the main result.

\begin{lemma}\label{lemma6} \cite{gm} Let $\mu$ be a positive Borel measure on $[0,1)$. Let $v$ be the Borel measure on $[0,1$ defined by
$$
dv(t)=\log\frac{e}{1-t}d\mu(t)
$$
Then the following statements are equivalent.
\begin{enumerate}
\item[(1)] $v$ is a Carleson measure.

\item[(2)] $\mu$ is a $1-$logarithmic $1-$Carleson measure.
\end{enumerate}
\end{lemma}

\begin{lemma}\label{lemma7} \cite{gm} Let $\mu$ be a positive Borel measure on $[0,1)$. Then the following statements are equivalent.
\begin{enumerate}
\item[(1)] The measure $\mu$ is a vanishing $1-$logarithmic $1-$Carleson measure.

\item[(2)] The operator $\H_\mu$ is compact on $\B$.

\item[(3)] The operator $\H_\mu$ is compact from $\B$ to BMOA.
\end{enumerate}
\end{lemma}

\begin{theorem}\label{theorem5} Let $\mu$ be an $1-$logarithmic $1-$Carleson measure on $[0,1)$. Then
\begin{equation}\label{gs2.4}
\|\H_\mu\|^{\B\rightarrow\B}_e\approx\|\H_\mu\|^{\B\rightarrow BMOA}_e\approx\limsup_{r\rightarrow 1^-}\frac{\mu(S([r,1)))\log\frac{e}{1-r}}{1-r}.
\end{equation}
\end{theorem}

\begin{proof} For any $f\in\B$, we have
$$
\|\H_\mu(f)\|^{\B\rightarrow\B}\lesssim\|\H_\mu(f)\|^{\B\rightarrow BMOA}.
$$
This gives that
$$
\|\H_\mu\|^{\B\rightarrow\B}_e\lesssim\|\H_\mu\|^{\B\rightarrow BMOA}_e.
$$
We now give the upper estimate of $\H_\mu$ from $\B$ to BMOA. Since $\mu$ is an $1-$logarithmic $1-$Carleson measure on $[0,1)$, the operator $\H_\mu$ from $\B$ to BMOA is bounded
by Theorem 2.8 of \cite{gm}. For any $0<r<1$, let the positive measure $\mu_r$ defined by \eqref{gs3.9}.
It is easy to check that $\mu_r$ is a vanishing $1-$logarithmic $1-$Carleson measure. We have $\H_{\mu_r}$ is compact from $\B$ to BMOA by Lemma  \ref{lemma7}. Then
\begin{equation}
\|\H_\mu\|^{\B\rightarrow BMOA}_e\leq\|\H_\mu-\H_{\mu_r}\|^{\B\rightarrow BMOA}=\inf_{\|f\|_\B=1}\|\H_{\mu-\mu_r}(f)\|_{BMOA}.
\end{equation}
By \eqref{gs2.1} we have
\begin{equation}\nonumber
\begin{split}
&\left|\int^{2\pi}_0\H_\mu(f_n)(re^{i\theta})\overline{g(e^{i\theta})}d\theta\right|\\
\leq&\int_{[0.1)}\left|f_n(t)\overline{g(rt)}\right|d(\mu-\mu_r)(t)\\
\leq&\int_{[0.1)}\left|\overline{g(rt)}\right|\log\frac{e}{1-t}d(\mu-\mu_r)(t)\\
\lesssim&\|v-v_r\|\|g\|_{H^1},
\end{split}
\end{equation}
where $dv(t)=\log\frac{e}{1-t}d\mu(t)$ and $dv_r(t)=\log\frac{e}{1-t}d\mu_r(t)$. The positive measure $v-v_r$ is a Carleson measure by Lemma \ref{lemma6}. The above estimate gives
$$
\|\H_\mu\|^{\B\rightarrow BMOA}_e\lesssim\limsup_{\lambda\rightarrow 1^-}\frac{\mu([\lambda,1))\log\frac{e}{1-\lambda}}{1-\lambda}.
$$

We will give the lower estimate for $\H_\mu$. Let $0<\lambda<1$ and
\begin{equation}\label{gs2.6}
f_\lambda(z)=\beta_\lambda\log^2\frac{e}{1-\lambda z},
\end{equation}
where $\beta_\lambda=\log^{-1}\frac{e}{1-\lambda^2}$. Then $\{f_\lambda\}$  is a bounded sequence in $\B$ and $\lim_{\lambda\rightarrow 1^{-}}f_\lambda(z)=0$ on any compact subset of $\D$. Since $f_\lambda\rightarrow 0$ weakly in $\B$,  we have $\|Kf_\lambda\|\rightarrow 0$ as $\lambda\rightarrow 1$ for any
compact operator $K$ on $\B$. Moreover
\begin{eqnarray*}
\|\H_\mu-K\|^{\B\rightarrow\B}&\geq&\|(\H_\mu-K)f_\lambda\|_{\B}\geq\|\H_\mu f_\lambda\|_\B-\|Kf_\lambda\|_\B.
\end{eqnarray*}
Note that $\H_\mu(f_\lambda)=I_\mu(f_\lambda)$. We have
\begin{equation}\nonumber
\begin{split}
\|\H_\mu(f_\lambda)\|_\B\geq&(1-\lambda^2)\left|\left(I_\mu(f_\lambda)\right)'(\lambda)\right|\\
\geq &(1-\lambda^2)\int_\lambda^1\frac{f_\lambda(t)}{(1-t\lambda)^2}d\mu(t)\\
\geq &\log\frac{e}{1-\lambda^2}\frac{\mu([\lambda,1))}{1-\lambda^2}.
\end{split}
\end{equation}
The above estimate shows that
$$
\|\H_\mu-K\|^{\B\rightarrow\B}_e\geq\limsup_{\lambda\rightarrow 1^-}\|\H_\mu f_\lambda\|_\B\gtrsim\limsup_{\lambda\rightarrow 1^-}\frac{\mu([\lambda,1))\log\frac{e}{1-\lambda}}{1-\lambda}.
$$
The proof is complete.
\end{proof}

\end{document}